\documentclass{article}

\usepackage{amsmath}
\usepackage{amssymb}
\usepackage{amsthm}
\usepackage{txfonts}
\usepackage{hyperref}
\usepackage{xcolor}

\newtheorem{remark}{Remark}
\newtheorem{lemma}{Lemma}
\newtheorem{theorem}{Theorem}
\newtheorem{corollary}{Corollary}

\renewcommand{\eqref}[1]{\hyperref[#1]{(\ref{#1})}}

\hypersetup{
  pdftitle    = {On norms and traces of the derivatives of the L2-projection error},
  pdfsubject  = {Numerical Analysis},
  pdfauthor   = {Torsten Lin\ss, Christos Xenophontos},
  pdfkeywords = {L2-projection, traces, Legendre polynomials},
  pdfcreator  = {pdflatex},
  pdfproducer = {LaTeX with hyperref},
    colorlinks,
    linkcolor={red!90!black},
    citecolor={blue!90!black},
    urlcolor={green!90!black}
}

\title{On norms and traces of the derivatives of~the~$L^2$-projection~error}

\date{\today}

\author{Torsten Lin\ss\thanks{Fakult\"at f\"ur Mathematik und Informatik,
        FernUniversit\"at in Hagen,
        Universit\"atsstra{\ss}e 11,
        58095 Hagen,
        Germany,
        \texttt{torsten.linss@fernuni-hagen.de}}
   \and Christos Xenophontos\thanks{Department of Mathematics \& Statistics,
        University of Cyprus, PO BOX 20537, Nicosia 1678, Cyprus,
        \texttt{xenophontos@ucy.ac.cy}.}
}

\numberwithin{equation}{section}

\newcommand{\pt}{\partial}
\newcommand{\RR}{\ensuremath{\mathbb{R}}}
\newcommand{\PP}{\ensuremath{\mathcal{P}}}
\newcommand{\QQ}{\ensuremath{\mathcal{Q}}}

\renewcommand{\phi}{\varphi}
\renewcommand{\rho}{\varrho}
\renewcommand{\theta}{\vartheta}
\newcommand{\scal}[2]{\left(#1,#2\right)}

\newcommand{\bigscal}[2]{\bigl(#1,#2\bigr)}

\newcommand{\NN}{\ensuremath{\mathbb{N}}}
\newcommand{\D}{\mathrm{\,d}}
\newcommand{\abs}[1]{\left|#1\right|}
\newcommand{\norm}[1]{\left\|#1\right\|}

\newcommand{\Qpnu}{\frac{1}{2^{2\nu}(\nu!)^2}\frac{(p+2)^{2\nu}}{2p+1}}

\begin{document}

\maketitle

\abstract{%
We derive bounds for the traces of derivatives of the $L^2$-projection error for functions in the Sobolev space $H^k$. The bounds are explicit in both the order of differentiation and the polynomial degree $p$. The analysis relies on the construction of a family of auxiliary polynomials involving weighted derivatives of \textsc{Legendre} polynomials. The univariate results are extended to two dimensions on rectangular elements by means of a tensor-product argument.

\emph{$L^2$-projection; trace estimates; \textsc{Legendre} polynomials;
   polynomial approximation; $hp$-finite element methods}

}

\section{Introduction}

Let \mbox{$\Omega \coloneqq (-1,1)$},
\mbox{$k \in \NN_0 \coloneqq \NN \cup \{0\}$},
and let $H^{k}(\Omega)$ be the usual Sobolev space of functions defined
on $\Omega$ whose weak derivatives up to order $k$ belong to~$L^2(\Omega)$, the space of square-integrable functions.
The $L^2$-projection of a function \mbox{$w\in H^{k}(\Omega)$} onto the space
$\PP_p(\Omega)$ of polynomials of degree at most $p$ on $\Omega$ is denoted by $\pi_p w \in \PP_p(\Omega)$ and is defined by
\begin{gather*}
  \scal{ \pi_p w - w}{v} = 0 \ \ \forall \; v \in \PP_p(\Omega),
\end{gather*}
where $\scal\cdot\cdot$ denotes the usual $L^2(\Omega)$ inner product.
Clearly,
\begin{gather*}
  \norm{\pi_p w - w}_{0} \le \norm{ w }_{0},
\end{gather*}
where $\norm{\cdot}_0$ is the norm in \mbox{$L^2(\Omega)$} induced by the scalar product $\scal\cdot\cdot$,
and we will write $\abs{w}_k\coloneqq \norm{w^{(k)}}_0, w\in H^k(\Omega)$.
It was shown in \cite[Thm.\,3.11]{MR1695813} that
\begin{gather}
  \label{eq:L2proj_bound}
  \norm{\pi_p w - w}_{0}^2
    \le \frac{(p+1-s)!}{(p+1+s)!} \abs{w}_s^2\,,
       \ \ s=0,\dots,\min\left\{p+1,k\right\}.
\end{gather}
Estimate \eqref{eq:L2proj_bound} shows that if $\abs{w}_s$ is uniformly bounded for \mbox{$s\to\infty$},
then \mbox{$\pi_p w \to w$} as \mbox{$p \to \infty$}.
In particular, if $w$ is an analytic function
satisfying $\vert w \vert_s \leq C \gamma^s s!$, for some positive constants $C,\gamma$ independent of $s$, then
the convergence rate is exponential.
This is achieved by selecting $s=\lfloor\lambda p\rfloor$ in \eqref{eq:L2proj_bound}, for a sufficiently small fixed $\lambda \in (0,1)$, depending on $\gamma$, and using Stirling's formula (see \cite{MR1695813} for more details).
In this article $s$ will be an integer,
but the results presented here may be extended to admissible non-integer values of $s$
using the $K$-method of interpolation (see, e.g., \cite{BL}).

For functions \mbox{$w\in H^{k}(\Omega)$},
\mbox{$k \in \NN$},
Houston et al.~\cite[Lemma~3.5]{MR1897953} established the following trace estimate for the error of the $L^2$-projection:
\begin{gather}
  \label{eq:L2proj_bound_trace}
  \abs{\left(\pi_p w - w\right)(\pm 1)}^2
    \le \frac{1}{2p+1} \frac{(p-s)!}{(p+s)!} \abs{w}_{s+1}^2\,,
       \ \ s=0,\dots,\min\left\{p,k-1\right\}.
\end{gather}
In this paper, we will generalize this result to give bounds on the traces of
\emph{derivatives of arbitrary order \mbox{$\nu\in\NN_0$}}
of the $L^2$-projection error.
\autoref{thm:main-1},
our main result, states that
\begin{multline*}
  \abs{\left(w - \pi_p w\right)^{(\nu)}(\pm1)}^2
    \le \norm{q_{p,\nu}}_0^2 \frac{(p-\nu-s)!}{(p-\nu+s)!}
        \abs{w}_{s+\nu+1}^2, \\
       s = 0,\dots,\min\left\{p - \nu,k-\nu-1\right\},
\end{multline*}
where $q_{p,\nu}$ belongs to a family of polynomials (defined in
\autoref{lem:qpn}) that satisfy the bound
\begin{gather}\label{eq:qpn_norm}
  \norm{q_{p,\nu}}_0^2 \le \Qpnu ,  \ \ 0\le\nu < p, 
\end{gather}
see \autoref{lem:norm-q-final} below.
For \mbox{$\nu=0$} and $p \geq 1$, our analysis yields a slight improvement
over~\eqref{eq:L2proj_bound_trace} (cf. \cite[Lemma~3.5]{MR1897953}).

Such results are useful tools in the analysis of various flavors of finite
element methods, such as (hybridized) discontinuous \textsc{Galerkin} (DG, HDG),
weak \textsc{Galerkin} (WG) or hybrid high-order methods (HHO),
applied to differential equations of fourth or higher order,
when bounds for the traces of derivatives are required. 
In such applications, the derivative order $\nu\in\mathbb N_0$ is determined by the differential equation and its numerical discretization and is therefore regarded as fixed (but arbitrary), while the polynomial degree $p$ is increased to improve accuracy. Motivated by this setting, we restrict our analysis to the case $p>\nu$.

The article is organized as follows.
In \autoref{sect:traces}, we derive a representation of the endpoint values of the derivative of order $\nu$ of the $L^2$-projection error as the inner product of an auxiliary polynomial and the derivative of order $\nu+1$ of the function being projected; see the proof of \autoref{lem:qpn}.
The $L^2$-norm of that polynomial is then bounded to yield~\eqref{eq:qpn_norm}.
\autoref{sect:2D} is dedicated to the extension of the results to two
dimensions, and \autoref{sect:concl} presents some concluding remarks.


\section{Traces of the derivatives of the $L^2$-projection error}
\label{sect:traces}

In preparation for what follows, we recall some properties of the
\textsc{Legendre} polynomials and their derivatives.

Let $P_j$ denote the \textsc{Legendre} polynomial of degree \mbox{$j\in\NN_0$}.
The \textsc{Legendre} polynomials form an orthogonal basis of $L^2(\Omega)$
with respect to the $L^2(\Omega)$ scalar product:
\begin{subequations}
  \label{prop-P_i}
\begin{gather}
   \scal{P_i}{v} = 0 \quad \forall \, v\in\PP_{i-1}, \ \ i\in\NN_0,
\end{gather}
where we have formally set $\PP_{-1}\coloneqq\{0\}$.
Moreover,
\begin{gather}
  \label{prop-P_i:ortho}
   \norm{P_i}_0^2 = \frac{2}{2i+1},
   \quad P_i(1)=1,
   \quad P_i(-1)=(-1)^i, \ \ i\in \NN_0.
\end{gather}
\end{subequations}
Let \mbox{$\rho \colon x \mapsto 1-x^2$}.
It is well known that the
derivatives of order $\nu$ of the \textsc{Legendre} polynomials, i.e.,
$P_i^{(\nu)}$, \mbox{$i\ge\nu$}, are orthogonal with respect to
the scalar product
\begin{gather*}
  \scal{v}{w}_\nu \coloneqq \int_{-1}^1 \rho (x)^\nu v(x) w(x) \D x.
\end{gather*}
We have
\begin{gather}
  \label{prop-P_i:Pp-ortho}
  \bigscal{P_i^{(\nu)}}{v}_\nu = 0, \quad \forall \, v\in\PP_{i-\nu-1}, \ \
     i\ge\nu,
  \intertext{and}
  \label{prop-P_i:Pp-norm}
    \bigscal{P_i^{(\nu)}}{P_i^{(\nu)}}_\nu =
       \int_{-1}^1 \rho (x)^\nu \bigl(P_i^{(\nu)}\bigr)^2(x) \D x
          = \frac{2}{2i+1}\frac{(i+\nu)!}{(i-\nu)!} , \ \ i\ge\nu.
\end{gather}
Furthermore, the following parity relation holds:
\begin{gather}
  \label{leg:parity}
  P_i^{(\nu)} (-x) = (-1)^{i+\nu} P_i^{(\nu)}(x), \ \ x\in\RR, \ \
     i,\nu\in\NN_0,
\end{gather}
while the \textsc{Legendre} differential equation implies
\begin{gather}
  \label{leg:norm-norm-deriv}
   \scal{P_i'}{P_i'}_1 = i(i+1) \scal{P_i}{P_i}
                       = i(i+1) \norm{P_i}_0^2,
       \ \ i\in\NN.
\end{gather}

\begin{lemma}
  \label{lem:norm-norm-deriv}
  Let \mbox{$i,k\in\NN$}, with \mbox{$1\le i \le k$}, and
  let $\chi\in\PP_k$ be orthogonal to $\PP_{i-1}$.
  Then
  \begin{gather*}
    \scal{\chi'}{\chi'}_1 \ge i(i+1) \norm{\chi}_0^2.
  \end{gather*}
\end{lemma}
\begin{proof}
  Because of the orthogonality assumed, the \textsc{Legendre} expansion of
  $\chi$ begins at $i$:
  \begin{gather*}
    \chi = \sum_{j=i}^k a_j P_j \quad\text{and}\quad
    \norm{\chi}_0^2 = \sum_{j=i}^k a_j^2 \frac{2}{2j+1}.
  \end{gather*}
  Hence, using \eqref{prop-P_i:Pp-ortho} and \eqref{prop-P_i:Pp-norm}
  we obtain
  \begin{align*}
    \int_{-1}^1\left(1-x^2\right) \chi'(x)^2 \D x
    =  \sum_{j=i}^k j(j+1) a_j^2 \frac{2}{2 j+1}
    \geq i(i+1) \sum_{j=i}^k a_j^2 \frac{2}{2 j+1}.
  \end{align*}
  The result follows.
\end{proof}
The following lemma bounds the endpoint values of the derivative of order $\nu$ of the $L^2$-projection error
by the product of the $L^2$-norm of an auxiliary polynomial and
 the $L^2$-norm of the derivative of order $\nu+1$  of the projected function.

\begin{lemma}
  \label{lem:qpn}
  Assume \mbox{$p,\nu\in\NN_0$}, \mbox{$\nu < p$}
  and \mbox{$u\in H^{\nu+1}(\Omega)$}.
  Let \mbox{$q_{p,\nu} \in \PP_{p+\nu+1}$} be
  such that
  \begin{gather}\label{eq:qpn:q_pnu}
    q_{p,\nu}(1)= 1, \ \ q_{p,\nu}(-1)=0 \quad\text{and}\quad
    q_{p,\nu}^{(i)}(\pm 1)=0\,, \ \ i=1,\dots,\nu,
    \intertext{and}
    \label{eq:qpn:q_pnu-ortho}
      \scal{q_{p,\nu}}{v} = 0\,, \ \ \forall \ v \in \PP_{p-\nu-1}.
  \end{gather}
  Then
  \begin{gather*}
     \abs{\left(u-\pi_p u\right)^{(\nu)}(\pm 1)}
        \le \norm{q_{p,\nu}}_0 \abs{u}_{\nu+1}.
  \end{gather*}
\end{lemma}
\begin{proof}
  Clearly, \mbox{$q_{p,\nu}^{(\nu+1)} \in \PP_p$},
  and therefore, by the definition of the $L^2$-projection,
  \begin{gather*}
     0 = \scal{u-\pi_p u}{q_{p,\nu}^{(\nu+1)}}.
  \end{gather*}
  Repeated integration by parts gives
  \begin{alignat*}{2}
     0 & = \sum_{i=0}^{\nu} (-1)^i\left. \left(u-\pi_p u\right)^{(i)} q_{p,\nu}^{(\nu-i)} \right|_{-1}^1
               + (-1)^{\nu+1} \int_{-1}^1 \left(u-\pi_p u\right)^{(\nu+1)} q_{p,\nu}\,,
  \intertext{and~\eqref{eq:qpn:q_pnu} yields}
    0  & = (-1)^{\nu} \left[ \left(u-\pi_p u\right)^{(\nu)}(1)
               -  \int_{-1}^1 \left(u-\pi_p u\right)^{(\nu+1)} q_{p,\nu}  \right]\\
       & =  (-1)^{\nu} \left[ \left(u-\pi_p u\right)^{(\nu)}(1)
               -\int_{-1}^1 u^{(\nu+1)} q_{p,\nu}\right].
  \end{alignat*}
  The latter follows from \eqref{eq:qpn:q_pnu-ortho}
  because \mbox{$\left(\pi_p u\right)^{(\nu+1)}\in \PP_{p-\nu-1}$}.
  Thus,
  \begin{gather*}
    \left(u-\pi_p u\right)^{(\nu)}(1) = \int_{-1}^1 u^{(\nu+1)} q_{p,\nu} \, .
  \end{gather*}
  The assertion of the lemma for the right endpoint
  follows from the \textsc{Cauchy-Schwarz} inequality.
  To show the result for the left endpoint,
  we repeat the argument using the reflected polynomial
  \mbox{$\tilde{q}_{p,\nu}(x)\coloneqq q_{p,\nu}(-x)$}.
\end{proof}

\begin{remark}\label{rem:sharp}
The above estimate is sharp.
At $x=1$ it holds with equality for any function $u$ with \mbox{$u^{(\nu+1)} = q_{p,\nu}$}, and
at $x=-1$ for any function $u$ with \mbox{$u^{(\nu+1)} = \tilde{q}_{p,\nu}$}.
\end{remark}

It remains to establish the existence of polynomials $q_{p,\nu}$ satisfying the above
properties. We start with $\nu=0$:
\begin{gather}
  \label{exist-q-0}
  q_{p,0} = \frac{P_p + P_{p+1}}{2} \in \PP_{p+1},
\end{gather}
which satisfies $q_{p,0}(1)=1$, $q_{p,0}(-1)=0$
and $\bigscal{q_{p,0}}{v} = 0$ for all $v\in\PP_{p-1}$,
cf.~\eqref{prop-P_i}.
The $L^2$ norm of $q_{p,0}$ can be computed using properties
of the \textsc{Legendre} polynomials (in particular, the orthogonality
of $P_p$ and $P_{p+1}$), yielding
\begin{gather}
  \label{bound:q_p0}
  \norm{q_{p,0}}_0^2
    =   \frac{1}{2}\left(\frac{1}{2p+1}+\frac{1}{2p+3}\right)
    =   \frac{2(p+1)}{(2p+1)(2p+3)}
    \le \frac{1}{2p+1},
\end{gather}
since \mbox{$2(p+1)/(2p+3)<1$} for all \mbox{$p\in\NN_0$}.
This result implies
\begin{gather*}
   \abs{\left(u-\pi_p u\right)(\pm 1)}^2
      \le \frac{1}{2p+1} \abs{u}_{1}^2,
\end{gather*}
which was established in \cite[the inequality following (3.24)]{MR1897953} as an
auxiliary result.
The proof in that paper is based on matching coefficients of the \textsc{Legendre}
expansions of both $u$ and $u'$.
Our approach avoids this step and also admits a generalization.

\begin{lemma}
  \label{lem:exist-q}
  Assume \mbox{$p,\nu\in\NN_0$}, \mbox{$\nu < p$}.
  The polynomial \mbox{$q_{p,\nu} \in \PP_{p+\nu+1}$} defined in
  \autoref{lem:qpn}
  is given by
  \begin{gather}
    \label{exist-q-gen}
      q_{p,\nu}(x)
        = \frac{1}{2^{\nu+1}\nu!}
          \int_{-1}^x \rho (s)^\nu \bigl(P_p+P_{p+1}\bigr)^{(\nu+1)}(s) \D s,
         \ \ 0\le\nu < p.
  \end{gather}
\end{lemma}
\begin{proof}
  For $\nu=0$, $q_{p,0}$ is given by~\eqref{exist-q-0}, which is equivalent to~\eqref{exist-q-gen}.

  For $0<\nu < p$, let $p$ and $\nu$ be fixed and set
  \begin{gather*}
    \tilde{q}(x) \coloneqq \frac{1}{2^{\nu+1}\nu!}
                   \int_{-1}^x \rho (s)^\nu \bigl(P_p+P_{p+1}\bigr)^{(\nu+1)}(s) \D s.
  \end{gather*}
  Thus,
  \begin{gather}\label{exist:aux}
    \tilde{q}'(x) = \frac{\rho (x)^\nu}{2^{\nu+1}\nu!}
                      \bigl(P_p+P_{p+1}\bigr)^{(\nu+1)}(x).
  \end{gather}
  By construction, \mbox{$\tilde{q}(-1)=0$} and \mbox{$\tilde{q}^{(i)}(\pm1) = 0$} for
  $i=1,\dots,\nu$, because $\rho^\nu$ has zeros of multiplicity $\nu$ at $\pm1$.
  It remains to verify that \mbox{$\tilde{q}(1)=1$} and that $\tilde{q}$ is
  orthogonal to $\PP_{p-\nu-1}$.

To this end, let $\chi\in\PP_{p-\nu}$ be arbitrary.
  Using integration by parts and the fact that
  $\bigl(\rho^\nu\chi\bigr)^{(i)}(\pm1)=0$ for $0 \le i \le \nu-1$ we obtain
  \begin{align*}
    {2^{\nu+1}\nu!}\int_{-1}^1 \tilde{q}'(x) \chi(x) \D x
    & = \int_{-1}^1 \rho (x)^\nu \bigl(P_p+P_{p+1}\bigr)^{(\nu+1)}(x) \chi(x) \D x \\
    & = (-1)^\nu \int_{-1}^1 \bigl(\rho ^\nu\chi\bigr)^{(\nu)}(x)
                               \bigl(P_p+P_{p+1}\bigr)'(x) \D x \\
    & = \left.(-1)^\nu \bigl(\rho ^\nu\chi\bigr)^{(\nu)}(x)
                               \bigl(P_p+P_{p+1}\bigr)(x)\right|_{-1}^1 \\
    & \qquad\qquad
            + (-1)^{\nu+1} \int_{-1}^1 \bigl(\rho ^\nu\chi\bigr)^{(\nu+1)}(x)
                               \bigl(P_p+P_{p+1}\bigr)(x) \D x.
  \end{align*}
  The integral vanishes, because $\bigl(\rho ^\nu\chi\bigr)^{(\nu+1)}\in\PP_{p-1}$.
  Furthermore, $\bigl(P_p+P_{p+1}\bigr)(1)=2$, $\bigl(P_p+P_{p+1}\bigr)(-1)=0$
  and $\bigl(\rho ^\nu\chi\bigr)^{(\nu)}(1) = (-2)^\nu \nu! \chi(1)$.
  Therefore,
  \begin{align}\label{eq:chi1}
    \int_{-1}^1 \tilde{q}'(x) \chi(x) \D x
      = \chi(1) \ \ \forall \, \chi \in \PP_{p-\nu}.
  \end{align}

  Taking \mbox{$\chi\equiv 1$} yields \mbox{$\tilde{q}(1)=1$},
  and it remains to verify that $\tilde{q}$ is orthogonal to $\PP_{p-\nu-1}$.
 Let \mbox{$v\in\PP_{p-\nu-1}$} be arbitrary and
  choose \mbox{$V\in\PP_{p-\nu}$} with \mbox{$V'=v$}.
  Then~\eqref{eq:chi1} gives
  \begin{gather*}
    \int_{-1}^1 \tilde{q}(x) v(x) \D x
      = \tilde{q}(x) V(x)\Bigr|_{-1}^1 - \int_{-1}^1 \tilde{q}'(x) V(x) \D x
      = V(1) - 0 - V(1) = 0.
  \end{gather*}
To establish uniqueness, let $q_1, q_2 \in \PP_{p+\nu+1}$ satisfy both \eqref{eq:qpn:q_pnu}
and \eqref{eq:qpn:q_pnu-ortho}, and set $r=q_1 - q_2$.
The homogeneous endpoint conditions imply that $r$ has zeros of multiplicity at
least $\nu+1$ at both endpoints. Therefore,
\begin{gather*}
r(x)=\rho (x)^{\nu+1} R(x)\; , \; R \in \PP_{p-\nu-1}.
\end{gather*}
By the orthogonality condition,
\begin{gather*}
0=(r, R)=\int_{-1}^1 \rho (x)^{\nu+1} R(x)^2 \D x
\end{gather*}
Since the weight is strictly positive in $(-1,1)$, we conclude that $R \equiv 0$ and hence $q_1 = q_2$.
Thus, \mbox{$\tilde{q}=q_{p,\nu}$}. 
\end{proof}
We are now in a position to state and prove the first main result of the article.
\begin{theorem}\label{thm:main-1} 
  Let $p, \nu \in \mathbb{N}_0, k \in \mathbb{N}, \nu \leq \min \{p-1, k-1\} $, \mbox{$w \in H^{k}(\Omega)$}
  and let \mbox{$\pi_p w \in \PP_p$} be its $L^2$-projection.
  Then, for $s \in\{0, \ldots, \min (p-\nu, k-\nu-1)\}$, the following estimate holds
  \begin{gather*}
    \abs{\left(w - \pi_p w\right)^{(\nu)}(\pm1)}
      \le \norm{q_{p,\nu}}_0 \left[ \frac{(p-\nu-s)!}{(p-\nu+s)!}\right]^{1/2} \abs{w}_{s+\nu+1},
  \end{gather*}
  with $q_{p,\nu}$ as in \autoref{lem:qpn}.
\end{theorem}
\begin{proof}Let \mbox{$W\in\PP_p$} be arbitrary with
  \begin{gather*}
    W^{(\nu+1)} = \pi_{p-\nu-1} \left(w^{(\nu+1)}\right).
  \end{gather*}
  Clearly, $\pi_p W = W$, since $W\in\PP_p$.
  Therefore,
  \begin{gather*}
    \abs{\left(w - \pi_p w\right)^{(\nu)}(\pm1)}
       = \abs{\left(w - W\right)^{(\nu)}(\pm1)
               - \left(\pi_p\left(w - W\right)\right)^{(\nu)}(\pm1)},
  \end{gather*}
 and \autoref{lem:qpn} implies
  \begin{gather*}
    \abs{\left(w - \pi_p w\right)^{(\nu)}(\pm1)}
     \le \norm{q_{p,\nu}}_0 \abs{w - W}_{\nu+1}
     =   \norm{q_{p,\nu}}_0 \norm{w^{(\nu+1)} - \pi_{p-\nu-1}\left(w^{(\nu+1)}\right)}_0\,.
  \end{gather*}
  Applying \eqref{eq:L2proj_bound}, with $p$ replaced by $p-\nu-1$, yields the desired result.
\end{proof}

\begin{remark}
  For \mbox{$\nu=0$} and $p \geq 1$, we recover \eqref{eq:L2proj_bound_trace}
  from Houston et al. \cite{MR1897953} -- with the slight improvement noted above.
\end{remark}

In the remainder of this section, we derive bounds on
$\norm{q_{p,\nu}}_0$.
\begin{lemma}
  \label{lem:norm-q}
  Assume \mbox{$p,\nu\in\NN_0$}, \mbox{$0 \le \nu < p$}.
  Then the $L^2$-norms of the polynomials \mbox{$q_{p,\nu} \in \PP_{p+\nu+1}$},
  defined in \autoref{lem:qpn}, satisfy
  \begin{gather*}
    \norm{q_{p,\nu}}_0^2
      \le \frac{1}{2^{2\nu+1}(\nu!)^2}
          \left[  \frac{(p+\nu+1)!}{(2p+1)(p-\nu+1)!}
                + \frac{(p+\nu+2)!}{(2p+3)(p-\nu+2)!}
          \right].
  \end{gather*}
\end{lemma}
\begin{proof}
  We split $q_{p,\nu}$ into its parity components:
  \begin{gather*}
    u(x) \coloneqq \frac{q_{p,\nu}(x)+(-1)^{p-\nu} q_{p,\nu}(-x)}{2}
    \quad \text{and} \quad
    z(x) \coloneqq q_{p,\nu}(x)-u(x).
  \end{gather*}
  Since $q_{p,\nu}$ is orthogonal to $\PP_{p-\nu-1}$,
  $u$ is orthogonal to $\PP_{p-\nu-1}$,
  while $z$ is orthogonal to $\PP_{p-\nu}$.
  This allows us to apply \autoref{lem:norm-norm-deriv}.
  Therefore,
  \begin{gather}\label{norm-q:aux}
    \norm{q_{p,\nu}}_0^2
      = \norm{u}_0^2 + \norm{z}_0^2
      \le \frac{\scal{u'}{u'}_1}{(p-\nu)(p-\nu+1)}
            + \frac{\scal{z'}{z'}_1}{(p-\nu+1)(p-\nu+2)}.
  \end{gather}

  By~\eqref{exist-q-gen} and the parity properties of $u'$, $z'$
  %
  and the derivatives of the \textsc{Legendre} polynomials, we have
  \begin{gather*}
    u' = \frac{\rho^\nu P_p^{(\nu+1)}}{2^{\nu+1}\nu!}
    \quad\text{and}\quad
    z' = \frac{\rho^\nu P_{p+1}^{(\nu+1)}}{2^{\nu+1}\nu!}.
  \end{gather*}
  Using $\rho^{2\nu+1}\le \rho^{\nu+1}$ on $[-1,1]$,
  together with \eqref{prop-P_i:Pp-norm}, we obtain
  \begin{align*}
    \scal{u'}{u'}_1
      & = \int_{-1}^1 \rho(x) u'(x)^2 \D x
        = \frac{1}{2^{2\nu+2} (\nu!)^2}
          \int_{-1}^1 \rho(x)^{2\nu+1} P^{(\nu+1)}_p(x)^2 \D x \\
      & \le \frac{1}{2^{2\nu+2} (\nu!)^2}
          \int_{-1}^1 \rho(x)^{\nu+1} P^{(\nu+1)}_p(x)^2 \D x
        = \frac{1}{2^{2\nu+2} (\nu!)^2}
          \frac{2}{2p+1} \frac{(p+\nu+1)!}{(p-\nu-1)!}.
  \end{align*}
  Similarly,
  \begin{gather*}
    \scal{z'}{z'}_1
        \le \frac{1}{2^{2\nu+2} (\nu!)^2}
          \frac{2}{2p+3} \frac{(p+\nu+2)!}{(p-\nu)!}.
  \end{gather*}
  Substituting these two bounds into~\eqref{norm-q:aux} yields
  the desired result.
\end{proof}

\begin{lemma}
  \label{lem:norm-q-final}
  Assume \mbox{$p, \nu\in\NN_0$} with \mbox{$0 \le \nu < p$}.
  Then the $L^2$-norms of the polynomials \mbox{$q_{p,\nu} \in \PP_{p+\nu+1}$}
  defined in \autoref{lem:qpn} satisfy
  \begin{gather*}
    \norm{q_{p,\nu}}_0^2 \leq 
      Q_{p,\nu} \coloneqq \Qpnu.
  \end{gather*}
\end{lemma}
\begin{proof}
  For $\nu=0$, the assertion follows from~\eqref{bound:q_p0}, so we consider \mbox{$p > \nu\ge1$}.
In order to apply \autoref{lem:norm-q}, we introduce
  \begin{gather*}
    \tilde{Q}_{p,\nu} \coloneqq
      \frac{(p+\nu+1)!}{(2p+1)(p-\nu+1)!}
                + \frac{(p+\nu+2)!}{(2p+3)(p-\nu+2)!}.
  \end{gather*}
  Elementary algebra gives
  \begin{gather*}
    \tilde{Q}_{p,\nu}
      = \frac{2\bigl(2(p+1)(p+2)-\nu\bigr)}{(2p+1)(2p+3)(p+2-\nu)}
        \cdot \frac{(p+\nu+1)!}{(p+1-\nu)!}.
  \end{gather*}
  Expressing the factorial quotient as a product, we obtain
  \begin{gather*}
    \frac{(p+\nu+1)!}{(p+1-\nu)!}
      = (p+2-\nu) (p+2) \prod_{k=1}^{\nu-1}\left((p+2)^2-k^2\right)
      \le (p+2-\nu) \left(p+2\right)^{2\nu-1}
  \end{gather*}
  and
  \begin{gather*}
    \tilde{Q}_{p,\nu}
      \le \frac{2\bigl(2(p+1)(p+2)-\nu\bigr)}{(2p+1)(2p+3)}
        \cdot \left(p+2\right)^{2\nu-1}.
  \end{gather*}
  Next, note that \mbox{$2(p+1)(p+2)-\nu \le (p+2)(2p+3)$}.
  Therefore,
  \begin{gather*}
    \tilde{Q}_{p,\nu}
      \le \frac{2 \left(p+2\right)^{2\nu}}{2p+1}.
  \end{gather*}
  Combining this estimate with \autoref{lem:norm-q} gives the result.
  
\end{proof}

Applying the above lemma to \autoref{thm:main-1}, we obtain our second
main result.
\begin{theorem}\label{thm:main-2}
  Let $p, \nu \in \mathbb{N}_0, k \in \mathbb{N}, \nu \leq \min \{p-1, k-1\} $, \mbox{$w \in H^{k}(\Omega)$}
  and let \mbox{$\pi_p w \in \PP_p$} be its $L^2$-projection.
  Then, for $s \in\{0, \ldots, \min (p-\nu, k-\nu-1)\}$, 
  \begin{gather*}
    \abs{\left(w - \pi_p w\right)^{(\nu)}(\pm 1)}^2
      \le \Qpnu \frac{(p-\nu-s)!}{(p-\nu+s)!} \abs{w}_{s+\nu+1}^2.
  \end{gather*}
\end{theorem}

\section{Extension to two dimensions}
\label{sect:2D}

Using the argument found in \cite[Lemma 3.6]{MR1897953},
the preceding one-dimensional results may be extended  to two dimensions
via a tensor-product argument.
To this end, let \mbox{$I=(-1,1), \widehat{\Omega}\coloneqq I \times I$} and define the space of polynomials of degree at most $p$ in each variable on $\widehat{\Omega}$, by
\begin{equation*}
  \QQ_{p}(\widehat{\Omega}) \coloneqq
     \Bigl\{ v\colon \widehat{\Omega} \to \RR \; | \;
        v(\cdot, z) \in \PP_{p}(I), \ v(z,\cdot) \in \mathcal{P}_{p}(I) \; \forall z\in(-1,1)
     \Bigr\}.
\end{equation*}
The $L^2(\widehat{\Omega})$-projection $\Pi_p v$ of any function
\mbox{$v\in L^2(\widehat{\Omega})$} onto $\QQ_{{p}}(\widehat{\Omega})$ is defined by
\begin{gather*}
   \int_{\widehat{\Omega}} \left(v - \Pi_p v\right) w = 0 \ \ \ \forall \ w\in \QQ_p(\widehat{\Omega}).
\end{gather*}
If \mbox{$v \in L^2(\widehat{\Omega})$}, then $\Pi_p v$ can be represented
as the tensor product of the one-dimensional $L^2$-projections,
in the two coordinate directions, $ \pi^{x}_p$ and $\pi^{y}_p$: 
\begin{equation*}
  \Pi_p \coloneqq \pi^{x}_p \pi^{y}_p = \pi^{y}_p \pi^{x}_p.
\end{equation*}
%
%
We also define the sets of the reference-element edges parallel
to the $x$  and $y$ axes, respectively, by
\begin{gather*}
  \widehat{\pt\Omega}_1 \coloneqq I \times\{ - 1, 1\} \quad \text {and} \quad
  \widehat{\pt\Omega}_2 \coloneqq \{ -1, 1\} \times I.
\end{gather*}
We adopt a slightly different (but self-explanatory) notation for the norms.

We have the following trace inequality.
We provide a short proof in order to keep all constants explicit.
\begin{lemma}
  \label{lem:trace}
  Let $w\in H^1(\widehat{\Omega})$.
  Then
  \begin{gather*}
    \norm{w}_{L^2(\widehat{\pt\Omega}_2)}^2
      \le 2 \norm{w}_{L^2(\widehat{\Omega})}^2
          + \norm{\pt_x w}_{L^2(\widehat{\Omega})}^2,
   \intertext{and}
    \norm{w}_{L^2(\widehat{\pt\Omega}_1)}^2
      \le 2 \norm{w}_{L^2(\widehat{\Omega})}^2
          + \norm{\pt_y w}_{L^2(\widehat{\Omega})}^2.
  \end{gather*}
\end{lemma}
\begin{proof}
  For almost every $y\in I$, we have
  \begin{align*}
    & \abs{w(-1,y)}^2 + \abs{w(1,y)}^2 \\
    & \qquad
       = \int_{-1}^{1}\frac{\D}{\D x}\bigl(x\, w(x,y)\bigr)^2 \D x
       = \int_{-1}^{1} w(x,y)^2 \D x
         + 2 \int_{-1}^{1} x\, w(x,y)\, \pt_x w(x,y) \D x \\
    & \qquad
       \leq \norm{w(\cdot,y)}_{L^2(I)}^2
         + 2 \norm{w(\cdot,y)}_{L^2(I)}
           \,\norm{\pt_x w(\cdot,y)}_{L^2(I)} \\
    & \qquad
       \leq 2 \norm{w(\cdot,y)}_{L^2(I)}^2
         + \norm{\pt_x w(\cdot,y)}_{L^2(I)}^2,
  \end{align*}
  where we used $\abs{x}\leq 1$, the \textsc{Cauchy-Schwarz} inequality
  and \textsc{Young}'s inequality.
  Integrating with respect to $y\in(-1,1)$ gives the desired bound. The second inequality follows by interchanging $x$ and $y$.
\end{proof}

We now establish the two-dimensional trace estimate.

\begin{theorem}\label{thm:2D}
  Let $p, \nu \in \mathbb{N}_0, k \in \mathbb{N}, \nu \leq \min \{p-1, k-1\} $, and assume $u\in H^k(\widehat{\Omega})$.
  Let $\Pi_pu\in\mathcal{Q}_p(\widehat{\Omega})$
  be the $L^2(\widehat{\Omega})$-projection.
  Then, for $ s\in \{0,\ldots,\min\{p-\nu,k-\nu-1\}\}$, 
  \begin{align}
    \norm{\pt_x^\nu(u-\Pi_pu)}_{L^2(\widehat{\pt \Omega}_2)}^2
      & \le \frac{2}{2^{2\nu} (\nu!)^2} \frac{(p+2)^{2\nu}}{2p+1} 
            \frac{(p-\nu-s)!}{(p-\nu+s)!}
            \norm{\pt_x^{s+\nu+1}u}_{L^2(\widehat{\Omega})}^2
              \notag \\
      & \qquad \qquad
                     + 2 \frac{(p-s)!}{(p+s+2)!}
                       \norm{\pt_x^{\nu} \pt_y^{s+1}u}_{L^2(\widehat{\Omega})}^2
            \notag \\
      & \qquad \qquad 
                     +  \frac{(p+1-s)!}{(p+1+s)!}
                       \norm{\pt_x^{\nu+1} \pt_y^{s}u}^2_{L^2(\widehat{\Omega})}
              , \label{eq:2D-x-new}
  \end{align}
  with an analogous bound for
  $\norm{\pt_y^\nu(u-\Pi_pu)}_{L^2(\widehat{\pt\Omega}_1)}^2$
  by interchanging $x$ and $y$.
\end{theorem}

\begin{proof}
Since the projections in the two coordinate directions commute, we have
 \begin{gather*}
u-\Pi_p u=\pi_p^y u-\pi_p^y \pi_p^x u+u-\pi_p^y u =\pi_p^y\left(\mathcal{I}-\pi_p^x\right) u+\left(\mathcal{I}-\pi_p^y\right) u ,
\end{gather*}
where $\mathcal{I}$ denotes the identity operator.
Differentiation with respect to $x$ commutes with $\pi_p^y$, so
\begin{gather*}
\partial_x^\nu\left(u-\Pi_p u\right)=\pi_p^y \partial_x^\nu\left(\mathcal{I}-\pi_p^x\right) u +\left(\mathcal{I}-\pi_p^y\right) \partial_x^\nu u .
\end{gather*}
For each $a \in \{ -1, 1 \}$, we have
\begin{gather*}
\pi_p^y \partial_x^\nu\left(\mathcal{I}-\pi_p^x\right) u(a, \cdot) \in \PP_p(I),
\end{gather*}
whereas
\begin{gather*}
\left(\mathcal{I}-\pi_p^y\right) \pt_x^\nu u(a, \cdot) \perp \PP_p(I) .
\end{gather*}
Therefore, by orthogonality on each vertical edge,
\begin{gather*}
\left\|\pt_x^\nu\left(u-\Pi_p u\right)\right\|_{L^2\left(\pt \widehat{\Omega}_2\right)}^2 
=\left\|\pi_p^y \partial_x^\nu\left(\mathcal{I}-\pi_p^x\right) u\right\|_{L^2\left(\pt \widehat{\Omega}_2\right)}^2
+\left\|\left(\mathcal{I}-\pi_p^y\right) \pt_x^\nu u\right\|_{L^2\left(\pt \widehat{\Omega}_2\right)}^2
=: S_1+S_2 .
\end{gather*}
By the stability of the $L^2$-projection,
\begin{gather*}
S_1 \leq\left\|\pt_x^\nu\left(\mathcal{I}-\pi_p^x\right) u\right\|_{L^2(\pt \widehat{\Omega}_2)}^2 .
\end{gather*}
Applying \autoref{thm:main-2} in the $x$-direction, integrating with respect to $y$, and summing over the two vertical edges yield
\begin{gather*}
  S_1 \leq \frac{2}{2^{2\nu} (\nu!)^2} \frac{(p+2)^{2\nu}}{2p+1}
           \frac{(p-\nu-s)!}{(p-\nu+s)!}
           \norm{\partial_x^{s+\nu+1} u}_{L^2(\widehat{\Omega})}^2.
\end{gather*}
Next, applying \autoref{lem:trace} and using the commutativity of $\pt_x$ and $\pi_p^y$, we obtain
\begin{gather*}
S_2 \leq 2\left\|\left(\mathcal{I}-\pi_p^y\right) \pt_x^\nu u\right\|_{L^2(\widehat{\Omega})}^2 +\left\|\left(\mathcal{I}-\pi_p^y\right) \pt_x^{\nu+1} u\right\|_{L^2(\widehat{\Omega}}^2 .
\end{gather*}
Applying \eqref{eq:L2proj_bound} in the $y$-direction, with regularity indices $s+1$ and $s$, respectively, yields
\begin{gather*}
S_2 \leq 2 \frac{(p-s)!}{(p+s+2)!}\left\|\pt_x^\nu \pt_y^{s+1} u\right\|_{L^2(\widehat{\Omega})}^2
+\frac{(p+1-s)!}{(p+1+s)!}\left\|\pt_x^{\nu+1} \pt_y^s u\right\|_{L^2(\widehat{\Omega})}^2 .
\end{gather*}
Combining the estimates for $S_1$ and $S_2$ establishes \eqref{eq:2D-x-new}.
The analogous estimate on $\pt \widehat\Omega_1$ follows by interchanging $x$ and $y$.
\end{proof}

An affine scaling argument yields the following corollary.
\begin{corollary}\label{cor:2D-rectangle}
  Let
  \begin{gather*}
    K\coloneqq (x_0,x_0+h_x)\times(y_0,y_0+h_y), \ \ h_x,h_y>0,
  \end{gather*}
  and define
  \begin{gather*}
    \pt K_1 \coloneqq (x_0,x_0+h_x)\times\{y_0,y_0+h_y\}, \ \
    \pt K_2 \coloneqq \{x_0,x_0+h_x\}\times(y_0,y_0+h_y).
  \end{gather*}
  Let \mbox{$p, \nu \in \mathbb{N}_0,  k \in \mathbb{N}, \nu \leq \min \{p-1, k-1\}$}, assume \mbox{$u\in H^k(K)$},
  and let $\Pi_p^K u\in\mathcal Q_p(K)$ denote the $L^2(K)$-projection.
  Then, for \mbox{$s\in\{0,\ldots,\min\{p-\nu,k-\nu-1\}\}$},
  %
  \begin{align}
    & \norm{\pt_x^\nu(u-\Pi_p^K u)}_{L^2(\pt K_2)}^2  \notag \\
    & \qquad
      \le \left(\frac{h_x}{2}\right)^{2s+1}
          \frac{2}{2^{2\nu} (\nu!)^2} \frac{(p+2)^{2\nu}}{2p+1}
          \frac{(p-\nu-s)!}{(p-\nu+s)!}
          \norm{\pt_x^{s+\nu+1}u}_{L^2(K)}^2 \notag\\
    & \qquad \qquad
       + 2
         \left(\frac{h_y}{2}\right)^{2s+2}
         \left(\frac{2}{h_x}\right)
         \frac{(p-s)!}{(p+s+2)!}
         \norm{\pt_x^\nu\pt_y^{s+1}u}_{L^2(K)}^2 \notag\\
    & \qquad \qquad
       + 
         \left(\frac{h_x}{2}\right)
         \left(\frac{h_y}{2}\right)^{2s}
         \frac{(p+1-s)!}{(p+1+s)!}
         \norm{\pt_x^{\nu+1}\pt_y^s u}_{L^2(K)}^2,
      \label{eq:2D-rectangle-x}
  \end{align}
  with an analogous bound for $\norm{\pt_y^\nu(u-\Pi_p^K u)}_{L^2(\pt K_1)}^2$
  by interchanging $x$ and $y$.
\end{corollary}

\begin{proof}
  Recall $\widehat{\Omega}=(-1,1)^2$
  and consider the affine mapping $F:\widehat{\Omega}\to K$ defined by
  \begin{gather*}
    x=x_0+\frac{h_x}{2}(1+\widehat{x}), \; \;
    y=y_0+\frac{h_y}{2}(1+\widehat{y}).
  \end{gather*}
  For $\widehat{u}=u\circ F$, the $L^2$-projections commute with the
  affine transformation, so that
  \begin{gather*}
    (\Pi_p^K u)\circ F=\Pi_p\widehat{u}.
  \end{gather*}
  Moreover,
  \begin{gather*}
    \pt_{\widehat{x}} =\frac{h_x}{2}\pt_x, \; \;
    \pt_{\widehat{y}} =\frac{h_y}{2}\pt_y,
  \end{gather*}
  and thus, for nonnegative integers $a,b$,
  \begin{equation}\label{scaling}
    \norm{\pt_{\widehat{x}}^a \pt_{\widehat{y}}^b\widehat{u}}
         _{L^2(\widehat{\Omega})}^2
     = \left(\frac{h_x}{2}\right)^{2a-1}
       \left(\frac{h_y}{2}\right)^{2b-1}
       \norm{\pt_x^a\pt_y^b u}_{L^2(K)}^2.
  \end{equation}
  Therefore, on the vertical sides, we have
  \begin{equation*}
    \norm{\pt_{\widehat{x}}^\nu (\widehat{u}-\Pi_p\widehat{u})}
         _{L^2(\widehat{\pt\Omega}_2)}^2 = \frac{(h_x/2)^{2\nu}}{h_y/2}
    \norm{\pt_x^\nu(u-\Pi_p^K u)}_{L^2(\pt K_2)}^2.
  \end{equation*}
  Applying \autoref{thm:2D} to $\widehat{u}$ and using \eqref{scaling}
  gives the three factors
  \begin{gather*}
    \left(\frac{h_x}{2}\right)^{2s+1},\;\;
    \frac{(h_y/2)^{2s+2}}{h_x/2},\;\;
    \frac{h_x}{2}\left(\frac{h_y}{2}\right)^{2s},
  \end{gather*}
which establishes \eqref{eq:2D-rectangle-x}.
\end{proof}

\section{Conclusions}
\label{sect:concl}

In this article, we have established estimates with fully explicit constants for the traces of derivatives of the $L^2$-projection error onto spaces of polynomials of degree at most $p$. The analysis relies on the construction of a family of auxiliary polynomials from weighted derivatives of \textsc{Legendre} polynomials and exploits their orthogonality and parity properties.

The resulting bounds are explicit in both the polynomial degree and the order of differentiation. The univariate estimates have also been extended to two dimensions on rectangular elements, with explicit dependence on the element dimensions.

These results are particularly relevant to the analysis of nonconforming numerical methods for higher-order boundary value problems, including problems of fourth and higher order (see, e.g., \cite{TRX}).


\paragraph{Note on the use of AI}
An earlier version of this paper was posted on arXiv in January 2026; see~\cite{LX}. The key result in that version was the representation of the traces of derivatives of the $L^2$-projection error, presented here as \autoref{lem:qpn}. However, the $L^2$-norms of the auxiliary polynomials $q_{p,\nu}$ were estimated using a different approach, based on their representation as linear combinations of certain primitives of \textsc{Legendre}  polynomials. In the process, explicit expressions for the $L^2$-norms of these primitives were obtained, generalizing the results in \cite{SW}.

During the revision of the manuscript, ChatGPT was used to assist in the analysis of the polynomials $q_{p,\nu}$. This led to the discovery of the representation~\eqref{exist-q-gen} as a generalization of~ \eqref{exist-q-0}. Subsequently, we were able to sharpen the bounds on $\Vert q_{p,\nu}\Vert_{0}$ and simplify the mathematical analysis.

The authors have independently reviewed and verified the resulting mathematical arguments and take full responsibility for the contents of the manuscript.

%
%
%



\begin{thebibliography}{1}

  \bibitem{BL} J.~Bergh and J.~L\"ofstr\"om,
    \newblock \emph{Interpolation Spaces: An Introduction}, Springer, 1976.

  \bibitem{MR1897953} P.~Houston, Ch.~Schwab, and E.~S\"uli,
    \newblock Discontinuous {$hp$}-finite element methods for
    advection-diffusion-reaction problems,
    \newblock {\em SIAM J. Numer. Anal.}, 39(6):2133--2163, 2002.

  \bibitem{TRX} T.~Lin\ss, A.~Raina and C.~Xenophontos,
    \newblock $rp$-WG finite element method for fourth-order singularly perturbed problems,
    \newblock in preparation, 2026.

  \bibitem{LX} T.~Lin{\ss} and C.~Xenophontos,
    \newblock On norms and traces of the derivatives of the $L^2$-projection error,
    \newblock \emph{arXiv:2601.06625v1}, 2026.

  \bibitem{SW} D.~Schötzau and T.~P.~Wihler,
    \newblock A posteriori error estimation for hp-version time-stepping methods
    for parabolic partial differential equations,
    \newblock \emph{Numer. Math.}, 115:475--509, 2010.

  \bibitem{MR1695813} Ch.~Schwab,
    \newblock \emph{{$p$}- and {$hp$}-Finite Element Methods:
    Theory and Applications in Solid and Fluid Mechanics},
    \newblock Clarendon Press, Oxford University Press, 1998.

\end{thebibliography}
\end{document}